\documentclass[11pt,a4paper, reqno]{amsart}
\usepackage{amsmath,amsfonts,amssymb,mathrsfs}
\usepackage{mathtools}
\usepackage{graphicx}
\usepackage{comment}
\usepackage{xfrac}
\usepackage{subfigure}
\usepackage{epstopdf}
\usepackage{pstricks}
\usepackage{pst-node}
\usepackage[linktocpage=true,colorlinks=true,linkcolor=blue,citecolor=red,urlcolor=magenta,pdfborder={0 0 0}]{hyperref}
\usepackage{fancyhdr}
\usepackage{mathscinet}
\usepackage{xcolor}
\usepackage[english]{babel}
\usepackage{mathscinet}

\usepackage{appendix}
\usepackage{cancel}

\definecolor{brickred}{rgb}{0.8, 0.25, 0.33}
\definecolor{forestgreen}{rgb}{0.13, 0.55, 0.13}
\definecolor{anna}{rgb}{0.01, 0.28, 1.0}

\newtheorem{theorem}{\bf Theorem}[section]
\newtheorem{lemma}[theorem]{\bf Lemma}

\newtheorem{definition}[theorem]{\bf Definition}
\newtheorem{remark}[theorem]{\bf Remark}
\newtheorem{proposition}[theorem]{\bf Proposition}

\newcommand{\R}{\mathbb{R}}

\def \L {\mathscr{L}}
\def \K {\mathscr{K}}

\def \epsilon {{\varepsilon}}

\def \l {{\lambda}}

\def \z {{\zeta}}

\def \phi {{\varphi}}

\def \loc {{\text{\rm loc}}}

\def\p{\partial}
\def \tilde {\widetilde}

\usepackage{mathtools}

\begin{document}
\title[Asymptotic lower bounds for the relativistic Fokker-Planck operator]{Harnack inequality and asymptotic lower bounds for the relativistic Fokker-Planck operator}

\author{Francesca Anceschi}
\address{Dipartimento di Ingegneria Industriale e Scienze Matematiche - Università Politecnica delle Marche: 
via Brecce Bianche, 12  60131 Ancona, Italy}
\email{f.anceschi@staff.univpm.it}
	
\author{Sergio Polidoro}
\address{Dipartimento di Scienze Fisiche, Informatiche e Matematiche - Università degli Studi di Modena e Reggio Emilia:
via G. Campi, 413/B 41125 Modena, Italy}
\email{sergio.polidoro@unimore.it}
	
\author{Annalaura Rebucci}
\address{Dipartimento di Scienze Matematiche, Fisiche e Informatiche - Università degli Studi di Modena e Reggio Emilia:
via G. Campi, 413/B 41125 Modena, Italy}
\email{annalaura.rebucci@unipr.it}
	
\date{\today}

\begin{abstract}
We consider a class of second order degenerate kinetic operators $\L$ in the framework of special relativity. We first describe $\L$ as an H\"ormander operator which is invariant with respect to Lorentz transformations. Then we prove a Lorentz-invariant Harnack type inequality, and we derive accurate asymptotic lower bounds for positive solutions to $\L f = 0$. As a consequence we obtain a lower bound for the density of the relativistic stochastic process associated to $\L$. 
	

		
\end{abstract}
	
\maketitle
	
\tableofcontents
\setcounter{equation}{0}\setcounter{theorem}{0}	
\section{Introduction}
\label{intro}

This work is devoted to the study of the following second order partial differential equation (PDE in the following)
\begin{equation} \label{oprel}
    \L f (p,y,t) =  {\sqrt{|p|^{2} + 1}}\, \nabla_p \cdot \left( D \, \nabla_p f\right) 
    - p \cdot \nabla_y f - \sqrt{|p|^{2} + 1} \, \p_t f = 0,
\end{equation}
where $(p,y,t) \in \R^{2d+1}$ and $D$ is the \textit{relativistic diffusion matrix} given by 
\begin{equation*}
 D=\frac{1}{\sqrt{|p|^{2} + 1}}\left( \mathbb{I}_d+ p\otimes p \right).
\end{equation*}
Here and in the following, $\mathbb{I}_d$ denotes the $d \times d$ identity matrix and $p\otimes p = \left( p_i p_j \right)_{i,j=1, \dots, d}$. In this context, a solution $f=f(p,y,t)$ to \eqref{oprel} denotes the density of particles in the phase space with momentum $p$ and position $y$, at time $t$.

We observe that $\L$ is a strongly degenerate differential operator, since only second order derivatives with respect to the momentum variable $p \in \R^d$ appear. However, the first order part of $\L$ induces a strong regularizing property. More precisely, $\L$ is hypoelliptic, namely every distributional solution $u$ to $\L u =f$ defined in some open set $\Omega \subset \mathbb{R}^{2d+1}$ belongs to $C^\infty(\Omega) $, whenever $f$ belongs to $C^\infty(\Omega)$. As a consequence, we only need to consider classical solutions to $\L f = 0$. Indeed, as we will show in Appendix \ref{sec-Horm}, we can write $\L$ in the form
\begin{equation} \label{hormander-op}
	\L := \sum_{j=1}^d X_j^{2} + X_{d+1},
\end{equation}
with
\begin{equation} \label{vector-fields}
 X_j = \sum_{k=1}^{d} \left( \delta_{jk} +\tfrac{p_j p_k}{1 + \sqrt{|p|^{2} + 1}}\right) \tfrac{\p  }{ \p p_k}, 
 \quad j=1, \dots, d, \quad  \text{and} \quad X_{d+1} = \sum_{k=1}^{d} c_k(p) X_k - Y, 
\end{equation}
where $c_1, \dots c_d$ are smooth functions and 
\begin{equation} \label{vector-fieldY}
		Y =  p \cdot \nabla_{y} + {\sqrt{|p|^{2} + 1}} \, \tfrac{\p}{\p t}.
\end{equation}
Moreover, $\L$ does satisfy the H\"ormander's rank condition \cite{H}
\begin{equation}\label{hormander}
	{\rm rank}\, {\rm Lie} \, \lbrace X_1,\ldots, X_d ,X_{d+1} \rbrace \,(p,y,t)=2d+1, 
	\qquad \forall (p,y,t) \in \mathbb{R}^{2d+1}, 
\end{equation} 
which is a well-known criterion for the hypoellipticity of an operator in the form \eqref{hormander-op}. We recall that in \eqref{hormander}, ${\rm Lie} \, \lbrace X_1,\ldots,X_d, X_{d+1} \rbrace$ denotes the vector space generated by the vector fields $\lbrace  X_1,\ldots, X_d , X_{d+1} \rbrace$ and by their commutators. 

\medskip

The aforementioned regularity property of operator $\L$ is related to a non-Euclidean structure on the space $\R^{2d +1}$ and its study needs to be addressed via an \emph{ad hoc} approach. In particular, as we will see in the sequel, $\L$ is the relativistic version of a kinetic diffusion operator and it is invariant with respect to Lorentz transformations. Moreover, $\L$ is the Kolmogorov equation of a suitable relativistic stochastic process $\left(P_{s}, Y_s, T_s \right)_{s \ge 0}$, that will be introduced in \eqref{RSDE} below. 
Our main result is a lower bound for the density of the stochastic process $\left(P_{s}, Y_s, T_s \right)_{s \ge 0}$. This is the first step in developing a systematic study of $\L$ in its appropriate framework of the PDEs theory. 
Indeed, our final aim is to extend the classical theory considered in \cite{APsurvey} to the relativistic case. In particular, we plan to prove asymptotic results such as \cite{LP, PaPo, LPP, AR-funsol} in this more general setting. 

\medskip

As we will see in Appendix \ref{sec-Horm}, the treatment of operator $\L$ in dimension $d>1$ involves cumbersome notation and computations. Thus, for the sake of simplicity, we restrict ourselves to the one-dimensional case, and we expect that the corresponding generalization of our main results to higher dimension does not imply substantial difficulties. In the one-dimensional case $\L$ writes in the following form
\begin{equation} \label{L1d}
 \L f (p,y,t) = {\sqrt{p^{2} + 1}}\, \tfrac{\partial}{\partial p} 
 \left({\sqrt{p^{2} + 1}}\, \tfrac{\partial f}{\partial p} \right) 
    - p \tfrac{\partial f}{\partial y} - \sqrt{p^{2} + 1} \, 
    \tfrac{\partial f}{\partial t},\qquad (p,y,t) \in \R^3,
\end{equation}
and takes the H\"ormander's form $\L = X^2 -Y$ if we set 
\begin{equation} \label{L1dh}
  X = {\sqrt{p^{2} + 1}}\, \tfrac{\partial}{\partial p} \qquad \text{and} \qquad 
  Y = p \tfrac{\partial f}{\partial y} + \sqrt{p^{2} + 1} \, \tfrac{\partial f}{\partial t}.
\end{equation}

\subsection{Physical interpretation}
Operator $\L$ is the relativistic version of the kinetic Fokker-Planck equation 
introduced by Kolmogorov \cite{K1} in $1934$
\begin{equation} \label{kolmo-base}
	\K f (p,y,t) = \frac{\p^{2} f }{ \p p^{2}} (p,y,t) - p \frac{\p f}{\p y} (p,y,t) - \frac{\p f}{\p t} (p,y,t) = 0
		\qquad (p,y,t) \in \R^{3}.
\end{equation}	
From the physical point of view, Fokker-Planck equations provide a continuous description of the dynamics of the 
distribution of Brownian test particles immersed in a fluid in thermodynamical equilibrium. More precisely, the distribution function $f$ of a test particle evolves according to the linear Fokker-Planck defined in \eqref{kolmo-base}, provided that the test particle is much heavier than the molecules of the fluid and that there is no friction. In particular, equation \eqref{kolmo-base} is the backward Kolmogorov equation of the Langevin process
\begin{equation}\label{langevin}
\begin{cases}
    & P_{t} = p_{0} + \sqrt{2} W_{t}, \\
	& Y_{t} = y_{0} + \int \limits_{0}^{t}P(s) ds,
\end{cases}
\end{equation}
where $\left(W_{t}\right)_{t \ge 0}$ denotes a $1-$dimensional Wiener process. We refer the reader to {\cite{APsurvey, Villani}, and the reference therein, for an exhaustive treatment of Fokker-Planck equations, and their applications.}

\medskip
In this article, we address a possible improvement of the model described in \eqref{kolmo-base} which is in accordance with special relativity.
Indeed, a questionable feature of \eqref{kolmo-base} is that its diffusion term $\frac{\p^2 f}{\p p^2}(p,y,t)$ operates with infinite velocity, as in classical mechanics the velocity is proportional to the momentum. In particular, it is known that, if we consider a continuous, non-negative and compactly supported initial distribution $f(p,y,0)$, then the unique non-negative solution $f(p,y,t)$ to the Cauchy problem relevant to \eqref{kolmo-base} is strictly positive for every positive $t$ (see, for instance, \cite[Theorem 3.3]{APsurvey}). In this scenario, there would be therefore a non-zero probability to find particles everywhere in space. This feature is clearly incompatible with the physical law that prevents particles from moving faster than light. To overcome this issue, we rely on the \emph{relativistic velocity}
\begin{equation}\label{rel-velocity}
v =\frac{p}{\sqrt{p^2+1}},
\end{equation} 
which clearly satisfies
\begin{equation*}
	\left\vert \frac{p}{\sqrt{p^{2} + 1}}\right\vert < 1 \quad \text{for every} \ p \in \R,
\end{equation*}
in accordance with the relativity principles\footnote{Here, we adopt a natural unit system with $c=1$, where $c$ is the speed of light.}. We consider the \emph{finite velocity Langevin process} analogous to \eqref{langevin}
\begin{equation}\label{SDE}
\begin{cases}
    & P_{t} = p_{0} + \sqrt{2} \int \limits_0^t \sqrt{P_s^2+1} dW_{s} \\
	& Y_{t} = y_{0} + \int \limits_{0}^{t} \frac{P_{s}}{\sqrt{P_{s}^{2} + 1} } ds,
\end{cases}
\end{equation}
and we recall that, by applying the relativistic It\^{o} calculus, Dunkel and H\"anggi find in \cite{DH2, DH} the Kolmogorov equation
 \begin{equation}\label{matrdiff}
\begin{split}
 \tilde{\L} f (p,y,t)=\frac{\p}{\p p}\left(
    \sqrt{p^{2} + 1}\, \frac{\p f}{\p p} (p,y,t)\right)- \frac{p}{\sqrt{p^2+1}} \frac{\p f}{\p y} (p,y,t) - \frac{\p f}{\p t} (p,y,t) = 0
   \end{split}
\end{equation}
satisfied by the density of the stochastic process $\left( P_t, Y_t\right)_{t \ge 0}$ in \eqref{SDE}. We refer the reader to \cite{Debbasch, DH2} for an overview to the relativistic theory of Brownian motions and corresponding relativistic kinetic equations. 

Alc\`antara and Calogero find the same equation \eqref{matrdiff} in \cite{AC} following a different approach, i.e. by merely requiring that some relevant properties of the non-relativistic equation are preserved in the relativistic setting. More precisely, as the non-relativistic operator is known to be Galilean invariant, the first requirement is the invariance with respect to the equivalent relativistic transformations, namely the Lorentz transformations. In addition, the authors of \cite{AC} impose that the relativistic Maxwellian distribution (or J\"{u}ttner distribution) $e^{-\gamma\sqrt{p^2+1}}, \gamma >0$, needs to be a stationary solution of equation \eqref{oprel} with friction, mirroring the fact that the Maxwellian distribution is a stationary solution of \eqref{kolmo-base} with friction. 
\medskip\\
We emphasize that $f$ is a solution to $\tilde{\L} f = 0$ if, and only if, it is a solution to $\L f = 0$ with $\L$ defined in \eqref{L1d}. We prefer to focus our attention on the differential operator $\L$ because it is invariant with respect to Lorentz transformations, as we will see in the following Subsection \ref{sec-inv}. We finally observe that \eqref{L1d} is the relativistic deterministic equation describing the density of the following stochastic process
\begin{equation}\label{RSDE}
\begin{cases}
    & P_{s} = p_{0} + \sqrt{2} \int \limits_0^s \sqrt{P_\tau^2+1}\, dW_{\tau}, \\
	& Y_{s} = y_{0} + \int \limits_{0}^{s} P_{\tau} d\tau,\\
	& T_{s} = t_{0} + \int \limits_{0}^{s} \sqrt{P_\tau^2+1}\, d\tau,
\end{cases}
\end{equation}
where the third component is the time, which is not an absolute quantity in the relativistic setting. 

\medskip

\subsection{Invariance properties}\label{sec-inv}
As it will be widely used in the sequel, we now focus on the invariance properties of operators $\L$ and $\K$.
As stated above, it is well known that $\K$ is invariant with respect to the Galilean change of variables
\begin{equation} \label{law-k} 
     (p_{0}, y_{0}, t_{0}) \circ_{\mathcal{G}} (p,y,t) = ( p_{0} + p, y_{0} + y + tp_{0} , t_{0} + t)
     \qquad \text{for every } \, \, (p_{0}, y_{0}, t_{0}), (p,y,t) \in \R^{3}.
\end{equation}
Indeed, if $ g(p, y, t) = f( p_{0} + p, y_{0} + y + tp_{0} , t_{0} + t)$ and $g^*(p, y, t) = f^*( p_{0} + p, y_{0} + y + tp_{0} , t_{0} + t)$, then
\begin{equation}\label{inv-K}
\K f =f^* \quad \iff \quad \K g=g^*, \quad \text{for every $(p_0,y_0,t_0) \in \R^3$}.
\end{equation}
In a natural way, operator $\L$ satisfies the relativistic analogue of property \eqref{inv-K}, i.e. it is invariant under Lorentz transformations. To show that, let us first summarize basic definitions and a few properties of those transformations. We recall that the relativistic momentum $p(t)$ and energy $E(t)$ of a particle with position $y(t)$ and velocity $v(t)=dy(t)/dt$ are given by
\begin{equation*}
p(t)=v\gamma(v(t))=\frac{v(t)}{\sqrt{1-v(t)^2}}, \qquad E(t)=\sqrt{p(t)^2+1}=\frac{1}{\sqrt{1-v(t)^2}}=\gamma(v(t)),\footnote{We here assume that the rest mass of the test particle is one.}
\end{equation*}
respectively, with $\gamma$ denoting the Lorentz factor 
\begin{equation*}
\gamma(v)=\frac{1}{\sqrt{1-v^2}}.
\end{equation*}
We combine time with position, and energy with momentum, to obtain the contravariant four-vectors\footnote{We use the term "four-vector" independently of the actual number of spatial dimensions.}
\[
\begin{pmatrix}
t\\
y
\end{pmatrix}, \qquad \text{and} \qquad \begin{pmatrix}
E\\
p
\end{pmatrix}.
\]
The above definitions refer to the intertial lab frame $\Sigma$, defined as the rest frame of the fluid. We now consider a second intertial frame $\tilde{\Sigma}$, moving with constant velocity $\beta$ with respect to $\Sigma$. According to Einstein's theory of special relativity, values of physical quantities in $\tilde{\Sigma}$ can be related to those in $\Sigma$ by means of the Lorentz transformations. In the one-dimensional case, the Lorentz transformation matrix reads as follows
\begin{equation*}
\Lambda(\beta)=\gamma(\beta)
\begin{pmatrix}
&1 &-\beta \\
&-\beta  &1
\end{pmatrix},
\end{equation*}
and its inverse is $\Lambda(-\beta)$. The matrices are representations of the Lorentz group acting on the four-vectors. The transformation law of arbitrary four-vectors is computed as follows
\begin{equation}\label{lorentz}
\begin{split}
\begin{pmatrix}
\tilde{t}\\
\tilde{y}
\end{pmatrix}
&=\gamma(\beta)\begin{pmatrix}
t-\beta y\\
y-\beta t
\end{pmatrix}, \qquad \begin{pmatrix}
\tilde{E}\\
\tilde{p}
\end{pmatrix}
=\gamma(\beta)\begin{pmatrix}
E-\beta p\\
p-\beta E
\end{pmatrix}\\
\begin{pmatrix}
t\\
y
\end{pmatrix}
&=\gamma(\beta)\begin{pmatrix}
\tilde{t}+\beta\tilde{y}\\
\tilde{y}+	\beta\tilde{t}
\end{pmatrix}, \qquad \begin{pmatrix}
E\\
p
\end{pmatrix}
=\gamma(\beta)\begin{pmatrix}
\tilde{E}+\beta\tilde{p}\\
\tilde{p}+\beta\tilde{E}
\end{pmatrix}.
\end{split}
\end{equation}
Let us consider the function, $\tilde{f}(\tilde{p},\tilde{y},\tilde{t})=f(p(\tilde{p}),y(\tilde{t},\tilde{y}),t(\tilde{t},\tilde{y}))$ that represents the one-particle phase space probability density function measured in the moving frame $\tilde{\Sigma}$, which, according to \cite{V}, transforms as a Lorentz scalar. If we set $g(p,y,t):=f(\tilde{p},\tilde{y},\tilde{t})$ and $g^*(p,y,t)=f^*(\tilde{p},\tilde{y},\tilde{t})$, then, applying the chain rule, we obtain 
\begin{align*}
\frac{\p g}{\p p}(p,y,t)&=\gamma\left(1-\frac{\beta p}{E} \right)\frac{\p f}{\p \tilde{p}}(\tilde{p},\tilde{y},\tilde{t})\\
\frac{\p g}{\p y}(p,y,t)&=\gamma\left(\frac{\p f}{\p \tilde{y}}(\tilde{p},\tilde{y},\tilde{t})-\beta \frac{\p f}{\p \tilde{t}}(\tilde{p},\tilde{y},\tilde{t})\right)\\
\frac{\p g}{\p t}(p,y,t)&=\gamma \left(-\beta\frac{\p f}{\p \tilde{y}}(\tilde{p},\tilde{y},\tilde{t})+ \frac{\p f}{\p \tilde{t}}(\tilde{p},\tilde{y},\tilde{t})\right).
\end{align*}
Then, the vector field $X$ defined in \eqref{L1d} is invariant with respect to Lorentz transformations, since
\begin{equation}\label{comp-diff}
\begin{split}
X(f(\tilde{p},\tilde{y},\tilde{t}))&=E \frac{\p g}{\p p}(p,y,t)= \gamma \left( E-\beta p\right)\frac{\p f }{\p p}(\tilde{p},\tilde{y},\tilde{t})\\
&=\tilde{E}\frac{\p f }{\p \tilde{p}}(\tilde{p},\tilde{y},\tilde{t})=(X f)(\tilde{p},\tilde{y},\tilde{t}).
\end{split}
\end{equation}
From \eqref{comp-diff}, it immediately follows that the diffusion term $X^2 f$ in \eqref{L1d} is also invariant with respect to Lorentz transformations. As far as we are concerned with the drift term $Y$ in \eqref{L1d}, we obtain
\begin{equation}\label{comp-drift}
\begin{split}
Y(f(\tilde{p},\tilde{y},\tilde{t}))&=p\frac{\p g}{\p y}(p,y,t)+E\frac{\p g}{\p t}(p,y,t)\\&=\gamma\left( p-\beta E\right)\frac{\p f}{\p \tilde{y}}(\tilde{p},\tilde{y},\tilde{t})+\gamma(-\beta p+E)\frac{\p f}{\p \tilde{t}}(\tilde{p},\tilde{y},\tilde{t})\\
&=\tilde{p}\frac{\p f}{\p \tilde{y}}(\tilde{p},\tilde{y},\tilde{t})+\tilde{E}\frac{\p f}{\p \tilde{t}}(\tilde{p},\tilde{y},\tilde{t})=(Yf)(\tilde{p},\tilde{y},\tilde{t}).
\end{split}
\end{equation}
In virtue of \eqref{comp-diff} and \eqref{comp-drift}, operator $\L$ is invariant with respect to the Lorentz transformations \eqref{lorentz}, i.e.
\begin{equation}\label{inv-L}
\L f =f^* \quad \iff \quad \L g=g^*, \quad \text{for every $(\tilde{p},\tilde{y},\tilde{t}) \in \R^3$}.
\end{equation}

Hence, owing to \eqref{inv-L}, operator $\L$ is invariant with respect to the following group operation on $\R^3$
\begin{equation}\label{group-law}
(p_0,y_0,t_0) \circ_{\mathcal{L}} (p,y,t)=\left(p\sqrt{p_0^2+1}+p_0 \sqrt{1+p^2},y_0+y\sqrt{p_0^2+1}+p_0t,t_0+t\sqrt{p_0^2+1}+p_0 y  \right).
\end{equation}
We remark that for small velocities $\sqrt{1+p_0^2} \approx 1$, and therefore \eqref{group-law} becomes precisely the non-relativistic composition law \eqref{law-k} for variables $p$ and $y$.

Moreover, $\mathbb{G}:=(\R^3, \circ_{\mathcal{L}})$ is a Lie group with identity $e$ and inverse $(p,y,t)^{-1}$ defined as:
\begin{equation*}
e=(0,0,0), \qquad (p,y,t)^{-1}=\left( -p,pt-\frac{ y}{\sqrt{p^2+1}}- \frac{p^2 y}{\sqrt{p^2+1}},-t\sqrt{p^2+1}+p y\right).
\end{equation*}
Then, in particular, we have that
\begin{align*}
&(p_0,y_0,t_0)^{-1} \circ_{\mathcal{L}} (p,y,t)\\
&\quad=\left(p\sqrt{p_0^2+1}-p_0 \sqrt{p^2+1},\sqrt{p_0^2+1}(y-y_0)-p_0(t-t_0),\sqrt{p_0^2+1}(t-t_0)-p_0(y-y_0) \right),
\end{align*}
so that \eqref{inv-L} is equivalent to 
\begin{equation}\label{inv-g}
f(p,y,t)=g((p_0,y_0,t_0)^{-1} \circ_{\mathcal{L}} (p,y,t)).
\end{equation}
\medskip

We conclude this Section with a remark concerning operator $\tilde{\L}$ in \eqref{matrdiff}. As already noticed, $\tilde{\L} f = 0$ if, and only if, $\L f = 0$. Moreover, $\tilde{\L}$ looks simpler than $\L$, as the derivative with respect to the time variable $\frac{\partial f}{\partial t}$ appearing  in $\L$ is multiplied by the coefficient $\sqrt{p^2+1}$, unlike $\tilde{\L}$. However, operator $\tilde{\L}$ is not invariant with respect to Lorentz transformations.

Indeed, Biagi and Bonfiglioli prove in \cite{BB} a general result for operators in the form \eqref{hormander-op},
where $X_1, \dots, X_{m+1}$ are H\"ormander's vector fields, with the property that, for every $z \in \R^N$, the integral curves
\footnote{The integral curve $\gamma: I \rightarrow \R^N$ of a vector field $X$ on $\R^N$ is defined by $\gamma'(s)=X(\gamma(s))$ for every $s \in I$.}
$\exp( t X_1)z$, \dots, $\exp( t X_{m+1})z$ are defined for every $t \in \R$. They prove that operator ${\L}$ is invariant with respect to the left translation of some Lie group $\mathbb{G}= \left( \R^{N}, \circ \right)$ if, and only if, the Lie algebra generated by $X_1, \dots, X_{m+1}$, as a linear subspace of the smooth vector fields in $\R^N$, has dimension $N$.

If we apply this condition to operator $\L = X^2 - Y$, with $X$ and $Y$ defined in \eqref{L1dh}, we find
\begin{equation*}
 [X,Y] =  \sqrt{p^2+1} \, \partial_y + p \,  \partial_t, \quad [X, [X,Y]] = Y
 \quad \text{and} \quad [Y, [X,Y]] = 0,
\end{equation*}
so that the dimension of the Lie algebra generated by $X$ and $Y$ equals 3, which is the dimension of the space $\R^3$. On the other hand, $\widetilde \L$ can be written in the form $\widetilde \L = {\widetilde X}^2 - \widetilde Y$, with
\begin{equation*}
 \widetilde X = \sqrt[4]{p^2+1} \, \partial_p \quad \text{and} \quad
 \widetilde Y =  \tfrac{p}{\sqrt{p^2+1}} \, \partial_p +\tfrac{p}{\sqrt{p^2+1}} \, \partial_y +  \, \partial_t,
\end{equation*}
and a direct computation shows that the dimension of the Lie algebra generated by $\widetilde X$ and $\widetilde Y$ is infinite. For this reason, we believe that $\L$ is the suitable relativistic generalization of \eqref{kolmo-base}.

\subsection{Main results} \label{secMainRes}
Our main result is a lower bound for the density of the stochastic process \eqref{RSDE}. This goal will be achieved via an approach in the framework of PDEs theory and therefore we state it in terms of \emph{fundamental solution} $\Gamma$. To this end, we recall the definition of $\Gamma$ below.
\begin{definition} 
\label{fun-sol-def}
We say that a function $\Gamma : \mathbb{R}^3 \times \mathbb{R}^3 \rightarrow \R $ is a {\rm fundamental solution} of operator $\L$ in \eqref{L1d} if it satisfies the following conditions:
\begin{itemize}
\item[1.] for every $(p_0,y_0,t_0)\in \R^3$, the function $(p,y,t)\mapsto \Gamma(p,y,t;p_0,y_0,t_0)$ belongs to $L^1_{\loc}(\R)\cap C^\infty(\R^3 \setminus \lbrace (p_0,y_0,t_0)\rbrace)$ and it is a classical solution to $\L u =0$ in $\R^3 \setminus \lbrace (p_0,y_0,t_0)\rbrace$;
\item[2.] for every $\phi \in C_b(\R^2)$, the function
\begin{equation*}
u(p,y,t)=\int_{\R^3} \Gamma(p,y,t;\xi,\eta,t_0)\phi(\xi,\eta)d\xi \, d\eta
\end{equation*}
is a classical solution to the Cauchy problem
\begin{equation*}
  \begin{cases}
  \L u(p,y,t)= 0, \qquad &\text{in  } \R^{2} \times ]t_0, + \infty[ \\
  f(p,y,t_0) = \phi \left(p, y\right), \qquad &\text{in  } \R^{2}.
  \end{cases}
\end{equation*}
\end{itemize}
\end{definition}

\medskip

In the statement of the following theorem, which is the  main result of this article, the function $\Psi$ is the \emph{value function} of a suitable optimal control problem and is defined in equation \eqref{defPSI} below.

\begin{theorem} \label{boundsL}
Let $\Gamma$ be the fundamental solution of $\L$ in \eqref{L1d}. Then for every $T > 0$ there exist three positive constants $\theta, c_T, C$ with $\theta < 1$, such that 
\begin{equation*}
	\Gamma (p_0,y_0,t_0;p_1,y_1,t_1) \ge \frac{c_T}{(t_0 - t_1)^{2}} \exp \left\{ - C \, \Psi \left(p_0, y_0, y_0; p_1, y_1, \theta^2t_1 + (1 - \theta^2)t_0 \right)  \right\}
\end{equation*}
for every $(p_0,y_0,t_0), (p_1,y_1,t_1) \in \R^3$ such that $0 < t_0 - t_1 < T$. The constants $\theta$ and $C$ only depend on $\L$, while $c_T$ also depends on $T$.
\end{theorem}

As far as the analogous upper bound is concerned, we believe it can be achieved by means of control theory in the same spirit of \cite{CPR, CintiPolidoro}. As this problem needs to be studied in a different framework, this issue will be addressed in a forthcoming paper. 
\medskip

This work is organized as follows. Section \ref{harnack} is devoted to the proof of an invariant Harnack inequality for solutions to $\L f = 0$. Section \ref{bounds} collects useful results on the optimal control problem associated to $\Psi$, while in Section \ref{setting} we give proof of our main result. Finally, in Appendix \ref{sec-Horm} we show how to place the higher dimensional operator \eqref{oprel} in H\"ormander's theory.

\setcounter{equation}{0}\setcounter{theorem}{0}
\section{Harnack inequality} \label{harnack}
\medskip
This section is devoted to the proof of a scale-invariant Harnack inequality for solutions to \eqref{oprel}, which is invariant with respect to Lorentz transformations. We introduce some notation necessary to state this result. For every positive $r$ we introduce the cylinders
\begin{equation} \label{eq-cyl-origin}
\begin{split}
	H_{r}(0) &:= \left\{ (p,y,t) \in \R^3 \mid \left| p\right| < r, 	\left|y \right| < r^3, -r^2 < t < 0 \right\},\\
	S_{r}(0) &:= \left\{ (p,y,t) \in \R^3 \mid \left| p\right| < r, 	\left|y \right| < r^3, -r^2 \leq t\leq-r^2/2 \right\}.
\end{split}
\end{equation}
Owing to \eqref{group-law}, for every $z_0=(p_0,y_0,t_0) \in \R^3$, we set
\begin{align}\label{eq-cyl-B-2}
H_r^{\mathcal{L}}(z_0):=z_0 \circ_{\mathcal{L}} H_r(0), \qquad S_r^{\mathcal{L}}(z_0):=z_0 \circ_{\mathcal{L}} S_r(0).
\end{align}
We are now in a position to state the following result.
\begin{theorem} \label{harnack-relativistic-vero}
There exist two constants $C_H>0$ and $\theta \in ]0,1[$, only depending on $\L$, such that
\begin{equation*}
	\sup \limits_{S^{\mathcal{L}}_{\theta r}(z_0)} f \, \le \, C_{H} f(z_0),
\end{equation*}
for every $z_0 \in \R^3, r \in ]0,1/2]$, and for every non negative solution $f$ to $\L f = 0$ in $H^{\mathcal{L}}_{r}(z_0)$. 
\end{theorem}

The proof of Theorem \ref{harnack-relativistic-vero} is obtained from the analogous Harnack inequality for the \emph{non-relativistic} kynetic operator $\widetilde \K$ acting as
\begin{equation} \label{kolmo}
	\widetilde \K f =a(p,y,t)\,\frac{\p^2 f}{\p p^2} + b(p,y,t) \frac{\p f}{\p p} - p  \frac{\p f}{\p y} -  \frac{\p f}{\p t} .
\end{equation}
In the following, $H_r^{\mathcal{G}}(z_0)=z_0 \circ_{\mathcal{G}} H_r(0)=\{ z \in \R^{3} : \, z = z_{0} \circ_{\mathcal{G}} \z, \, \z \in  H_r(0) \}$ and $S_r^\mathcal{G}(z_0)=z_0 \circ_{\mathcal{G}} S_r(0)=\{ z \in \R^{3} : \, z = z_{0} \circ_{\mathcal{G}} \z, \, \z \in  S_r(0) \}$, where $\circ_{\mathcal{G}}$ is the composition law defined in \eqref{law-k}. We here recall the statement of the Harnack inequality for classical solutions to \eqref{kolmo} proved in \cite{PODF}.
\begin{theorem} \label{harnack-kolmo-GIMV}
Suppose that the coefficients $a$ and $b$ in \eqref{kolmo} are continuous functions satisfying
\begin{itemize}
\item[\textbf{(H)}] There exist two constants $\l^-, \l^+> 0$ such that
\begin{equation*}
	\l^- \le a(p,y,t) \le \l^+,\\\qquad |b(p,y,t)| \le \l^+ 
	\qquad \text{for every} \ (p,y,t) \in \R^{3}.
\end{equation*}
\end{itemize}
Then there exist two constants $C_H>0$ and $\theta \in ]0,1[$, only depending on $\l^-$ and $\l^+$ such that
\begin{equation*}
	\sup \limits_{S^{\mathcal{G}}_{\theta r}(z_0)} f \, \le \, C_{H} f(z_0).
\end{equation*}
for every $z_0 \in \R^3, r \in \left]0, \frac12\right[$, and for every non negative solution $f$ to $\widetilde \K f = 0$ in $H^{\mathcal{G}}_{r}(z_0)$.
\end{theorem}
Several proofs of Theorem \ref {harnack-kolmo-GIMV} are available in literature. We refer to \cite{APsurvey} and to its bibliography for a survey on the classical theory to operator \eqref{kolmo}. In recent years, researchers have focused on developing the weak regularity theory for solutions to equations of the type \eqref{kolmo}. Among the most recent results, we recall the Harnack inequality proven by Golse, Imbert, Mouhot and Vasseur in \cite{GIMV} concerning weak solutions to divergence form operators with measurable coefficients. We also refer the reader to \cite{AEP} for a geometric statement of the Harnack inequality, and to \cite{AR-harnack} for an invariant Harnack inequality for the more general Kolmogorov operator.

\subsection{Change of variable}
In order to prove Theorem \ref{harnack-relativistic-vero}, we perform an appropriate change of variable that allows us to write solutions to $\tilde{\L} f =0$ (and therefore to $\L f = 0$) in the form \eqref{kolmo}. To this end, we denote by $\varphi$ the function defined as follows
\begin{equation} \label{eq-varphi}
 \varphi(p) := \frac{p}{\sqrt{1 + p^2}}, \qquad p \in \R,
\end{equation}
where we remark that $\phi(p)$ is actually the relativistic velocity defined in \eqref{rel-velocity}. By a direct computation, we easily obtain
\begin{equation}\label{der-varphi}
 	\varphi'(p) := \frac{1}{(1 + p^2)^{3/2}}, \qquad 
 	\varphi''(p) := - \frac{3p}{(1 + p^2)^{5/2}}.
\end{equation}
Moeover, it is easy to verify the function
\begin{equation}\label{psi}
\psi(x) := \frac{x}{\sqrt{1 - x^{2}}}
\end{equation}
is the inverse of $\varphi$, and the following identity holds
\begin{equation} \label{eq-varphi-id}
 1-\varphi^2(p) = \frac{1}{1 + p^2}, \qquad p \in \R.
\end{equation}
We are now in a position to state and prove the following preliminary result.

\begin{lemma} \label{equivalence} 
Let $f$ be a solution to $ \L f = 0$. For every $(x,y,t) \in ]-1,1[ \times \R^2$ , we define the function
\begin{equation} \label{eq-def-u}
 u(x,y,t) := f\left( \psi\left( x\right),y, t \right),
\end{equation}
where $\psi$ was defined in \eqref{psi}.
Then $u$ is a solution to the following equation
\begin{equation} \label{kolmorel}
	 \frac{\p u}{\p t} (x,y,t) + x  \frac{\p u}{\p y} (x,y,t)  
	 = \left(1 - x^2 \right)^{5/2} \frac{\p^2 u }{ \p x^2} (x,y,t) 
	 -2x \left(1 - x^2\right)^{3/2} \frac{\p u }{ \p x} (x,y,t).
\end{equation}
\end{lemma}
\begin{proof} By inverting the change of variable in \eqref{eq-def-u} we find that 
\begin{equation}\label{invchanvar}
x= \phi(p), 
\end{equation}
and therefore
\begin{equation*}
 f(p,y,t) = u \left(  \phi(p),  y, t \right).
\end{equation*}
Hence, from the chain rule it follows immediately 
\begin{equation}\label{eq-drif}
 \frac{p}{\sqrt{1 + p^2}} \frac{\p f}{\p y}(p,y,t) +\frac{\p f}{\p t}(p,y,t) = 
 	\varphi(p) \frac{\p u }{ \p y} ( x,y, t ) + \frac{\p u }{ \p t} ( x,y,t ).
\end{equation}
Moreover, exploiting identities \eqref{der-varphi}, \eqref{eq-varphi-id} and \eqref{invchanvar}, we obtain
\begin{equation} \label{eq-coeff-c}
\begin{split}
 	 \frac{\p f}{\p p}(p,y,t)  \, &= \varphi'(p) \frac{\p u }{ \p x} ( x,y, t ), \\
	 \frac{\p^2 f}{\p p^2}(p,y,t)  \, &=  (\varphi'(p))^2 \frac{\p^2 u }{ \p x^2} ( x,y, t )
	 +  \varphi''(p) \frac{\p u }{ \p x} ( x,y, t )\\
	&= \frac{1}{(1+p^2)^3}\frac{\p^2 u }{ \p x^2} (x,y, t )-\frac{3p}{(1+p^2)^{5/2}}\frac{\p u }{ \p x} ( x,y, t )\\
	&=\left( 1-\varphi^2(p)\right)^3\frac{\p^2 u }{ \p x^2} (  x,y, t )-3\varphi(p)
	\left( 1-\varphi^2(p)\right)^2\frac{\p u }{ \p x} ( x,y, t )\\
	&=\left( 1-x^2\right)^3\frac{\p^2 u }{ \p x^2} (  x, y, t )-3x\left( 1-x^2\right)^2\frac{\p u }{ \p x} ( x, y, t ).
\end{split}
\end{equation}
As a consequence, the diffusion term in equation \eqref{matrdiff} becomes	
\begin{equation}\label{eq-diff}
\begin{split}
&\sqrt{p^2+1}\frac{\p^2 f}{\p p^2}(p,y,t) +\frac{p}{\sqrt{p^2+1}} \frac{\p f}{\p p}(p,y,t)\\&\quad=  \left( 1-x^2\right)^{5/2}\frac{\p^2 u }{ \p x^2} (  x, y, t )-3x\left( 1-x^2\right)^{3/2}\frac{\p u }{ \p x} ( x, y, t )+x\left( 1-x^2\right)^{3/2}\frac{\p u }{ \p x} ( x, y, t )\\
&\quad=  \left( 1-x^2 \right)^{5/2}\frac{\p^2 u }{ \p x^2} (  x, y, t )-2x\left( 1-x^2\right)^{3/2}\frac{\p u }{ \p x} ( x, y, t ).
\end{split}
\end{equation}

The claim then follows combining \eqref{eq-drif} and \eqref{eq-diff} and observing that $\L f=0$ if and only if $\tilde{\L} f =0$.
\end{proof}
%

\subsection{Proof of Theorem \ref{harnack-relativistic-vero}}

We observe that the operator appearing in \eqref{kolmorel} writes in the form \eqref{kolmo} if we choose 
\begin{align} \label{def-ab}
a(x,y,t) =  \left(1 - x^2\right)^{5/2}, \qquad \text{and} \qquad b(x,y,t) = -2x \left(1 - x^2\right)^{3/2}.
\end{align}
Moreover, we remark that condition {\bf (H)} is satisfied only on compact subsets of $]-1,1[ \times \R^2$ as we have 
$$
\inf\limits_{-1 < x < 1} a(x,w,t) = 0.
$$ 

\begin{proof}[Proof of Theorem \ref{harnack-relativistic-vero}]
Since $\L$ is invariant with respect to the Lorentz transformations \eqref{group-law},
we first restrict ourselves to the case where $z_0=(p_0,y_0,t_0)=(0,0,0)$. As a consequence, owing to \eqref{invchanvar}, we have also $(x_0,y_0,t_0)=(0,0,0)$. We then observe that for every $p \in \left[ -\frac{1}{2},\frac{1}{2}\right]$, there holds
\begin{equation}  \label{constants-h}
 \left(\frac{1}{2} \right)^{5/2} \le\left( \frac{1}{1+p^2}\right)^{5/2} \le 1.
\end{equation}

We now apply the change of variable \eqref{eq-def-u} and we observe that $a(0,0,0)=1$, where $a(x,y,t)$ is the coefficient in \eqref{def-ab}. Keeping in mind that $ x = \varphi(p)$, we find that for every point $(p,y,t) \in H^{\mathcal{L}}_r(0)$, the following inequality 
\begin{equation} \label{eq-xQ}
  | x| = \left|\frac{p}{\sqrt{1+p^2}}\right| \le |p| \le r
\end{equation}
holds true.
Thus, $(  \varphi(p),  y, t) \in H^{\mathcal{G}}_r{(0)}$ for every $(p,y,t) \in H^{\mathcal{L}}_r(0)$ and for every $r \in \left]0,\frac{1}{2}\right[$. Furthermore, from definition \eqref{def-ab} it follows that 
\begin{equation*}
  a( \varphi(p),  y, t) = \left(\frac{1 }{1 + p^2}\right)^{5/2},
\end{equation*}
and therefore 
 \begin{equation*}
 \left(\frac{1}{2}\right)^{5/2} \le a( x, y, t) \le 1, \qquad
  |b(  \varphi(p),  y, t)| = 2 \frac{|p|}{\sqrt{1 + p^2}}\frac{1}{{(1 + p^2)}^{3/2} }\le 2 ,
\end{equation*}
for every $(x,y,t) = (  \varphi(p),  y, t)$ with $(p,y,t) \in H^{\mathcal{L}}_r(0)$. 
Thus, the coefficients appearing in \eqref{kolmorel} satisfy assumption \textbf{(H)} with $\lambda^-=\left(\frac{1}{2}\right)^{5/2}$ and $\l^+=2$. Since $( \varphi(p), y, t) \in H^{\mathcal{G}}_r{(0)}$ for every $(p,y,t) \in S^{\mathcal{L}}_r(0)$ and for every $r \in \left]0,\frac{1}{2}\right[$, our claim is proven for $z_0=(0,0,0)$. 

In order to prove our claim for an arbitrary point $z_0 \in \R^3$, we rely on the invariance of $\L$ with respect to \eqref{group-law}. In particular, we apply the Lorentz change of variables \eqref{inv-g} to a solution $f$ to $ \L f =0$ in $H^{\mathcal{L}}_r(z_0)$ and we observe that the function $g$ in \eqref{inv-g} is a solution to $\L g = 0$ in $H^{\mathcal{L}}_r(0)$. Then the Harnack inequality holds for $g$ and implies that
\begin{equation*}
f(p,y,t) = g((p_0,y_0,t_0)^{-1} \circ_{\mathcal{L}} (p,y,t)) \leq C_H g(0,0,0)=C_H f(p_0,y_0,t_0),
\end{equation*}
where $C_H$ does not depend on $z_0$. This concludes the proof. 
\end{proof}

\begin{remark}
We observe that the ``cylinders" defined in \eqref{eq-cyl-B-2} are the most natural geometric sets which can be defined starting from \eqref{eq-cyl-origin} and using the invariance group of $\L$. Finally, let us remark that, in virtue of \eqref{group-law}, the sets \eqref{eq-cyl-B-2} can be explicitely computed as follows
\begin{equation}\label{cyl-explicit}
\begin{split}
	H_{r}^{\mathcal{L}}(z_{0}) 
	&:= \left\{ (p,y,t) \in \R^3 \mid \left| p\sqrt{1+p_0^2} - p_0\sqrt{1+p^2}\right| < r, \right.\\
	&\quad\quad\quad\left.	\left|\sqrt{1+p_0^2}(y - y_0) - p_0 (t-t_0) \right| < r^3, 
	-r^2 < \sqrt{1+p_0^2}(t-t_0)-p_0(y - y_0) < 0 \right\},\\
	S_{r}^{\mathcal{L}}(z_{0}) &:= \left\{ (p,y,t) \in \R^3 \mid \left| p\sqrt{1+p_0^2} - p_0\sqrt{1+p^2}\right| < r, \right.\\
		&\quad\quad\quad\quad\quad\quad\quad \left.	\left|\sqrt{1+p_0^2}(y - y_0) - p_0 (t-t_0) \right| < r^3, \right. \\
		&\quad\quad\quad\quad\quad\quad\quad \left. -r^2 \leq \sqrt{1+p_0^2}(t-t_0)-p_0(y - y_0) \leq-r^2/2 \right\}.
\end{split}
\end{equation} 
However, we do not need the explicit expression \eqref{cyl-explicit} in our treatment, as it is sufficient to rely on definition \eqref{eq-cyl-B-2} and on the invariance properties of $\L$.
\end{remark}

\setcounter{equation}{0}\setcounter{theorem}{0}
\section{Optimal control problem} \label{bounds}
This section is devoted to the proof of Theorem \ref{boundsL}. In order to provide a clear treatment, we first recall some fundamental notions from control theory and prove an equivalent statement of the Harnack inequality, more suitable 
to the construction of Harnack chains (see Proposition \ref{prop-cor}). We then prove an estimate for positive solutions to $\L f = 0$ (see Proposition \ref{prop-Phi}) depending on the
norm of the control. Finally, we conclude this section with a preliminary study of the optimal control problem associated to $\L$.

\subsection{$\L$-admissible paths and Harnack chains}
Along with the Harnack inequality Theorem \ref{harnack-relativistic-vero}, the main tool in the proof of our asymptotic estimates for the fundamental solution are \textit{Harnack chains}, whose definition we recall below.
\begin{definition}[Harnack chain]
Let $\Omega$ be an open subset of $\R^3$. We say that a finite set of points $\lbrace z_0,z_1,\ldots,z_k\rbrace \in \Omega$ is a {\rm Harnack chain} connecting $z_0$ to $z_k$ if there exist positive constants $C_1, \ldots,C_k$ such that
\begin{equation*}
u(z_j) \leq C_j u(z_{j-1})\qquad j=1,\ldots,k,
\end{equation*}
for every positive solution $u$ to $\L u=0$.
\end{definition}

In the present setting, we construct Harnack chains by connecting points belonging to appropriate trajectories, which naturally substitute segment lines in our non- Euclidean setting and are defined as follows.

\begin{definition}[$\L$--admissible path]
A curve $\gamma(s)=(p(s),y(s),t(s)):[0,T] \rightarrow \R^3$ is said to be {\rm a $\L$--admissible path} if it is absolutely continuous and solves 
the following differential equation
\begin{equation} \label{eq-admiss}
  \gamma'(s) = \omega(s) X(\gamma(s)) + Y(\gamma(s)), 
\end{equation}
for almost every $s \in [0,T ]$, where $X$ and $Y$ are defined in \eqref{L1dh}. Moreover, we say that $\gamma$ steers $(p_0,y_0,t_0)$ to $(p_1,y_1,t_1)$, with $t_0 > t_1$, if 
\begin{equation} \label{eq-admiss-pt}
  \gamma(0)=(p_0,y_0,t_0), \qquad  \gamma(T)=(p_1,y_1,t_1).
\end{equation}
\end{definition}
In the definition of $\L$--admissible path we assume $\omega \in L^2([0,T])$ and we refer to the function $\omega$ as the control of problem \eqref{eq-admiss}. Let us remark that, owing to \eqref{L1dh}, equation \eqref{eq-admiss} can be explicitly written as follows
\begin{equation} \label{expl-diffeq} 
\begin{split}
 &\quad\qquad\qquad\qquad\begin{cases}
 		p'(s)&=\omega(s)\sqrt{p^2(s)+1},  \\
 		y'(s) &=- p(s),\\
 		t'(s) &=-\sqrt{p^2(s)+1},
 	\end{cases} 
\end{split}
\end{equation}
for almost every $s \in [0,T]$.

Moreover, we observe that such optimal control problem is invariant with respect to the group operation \eqref{group-law}. Indeed, let us consider a control $\omega(\cdot)$ steering $(p_0,y_0,t_0)$ to $(p_1,y_1,t_1)$ with trajectory $(p(s),y(s),t(s))$. Then, it is easy to prove the trajectory $(\tilde{p}(s),\tilde{y}(s),\tilde{t}(s)):=(p_0,y_0,t_0)^{-1}\, \circ_{\mathcal{L}}\, (p(s),y(s),t(s))$ is a solution to \eqref{eq-admiss}-\eqref{eq-admiss-pt} with the same control $\omega(\cdot)$. Additionally, the newly defined trajectory $(\tilde{p}(s),\tilde{y}(s),\tilde{t}(s))$ satisfies the properties $(\tilde{p}(0),\tilde{y}(0),\tilde{t}(0))=(0,0,0)$.

Finally, we introduce the standard definition of attainable set from control theory.

\begin{definition}[Attainable set]
For every $z_0 \in \Omega \subset \R^3$, the {\rm attainable set $\mathscr{A}_{z_0}$ of $z_0$ in $\Omega$} is defined as follows
\begin{align*}
\mathscr{A}_{z_0}=\left\{ z \in \Omega \, : \, \textit{there exists $\bar{t} \in \R^+$ and a $\L$-admissible path $ \gamma:[0,\bar{t}]\rightarrow \Omega$ }	\right.	\\
\left. \qquad \textit{ such that $\gamma(0)=z_0$, $\gamma(\bar{t})=z$}\right\}
\end{align*}
\end{definition}

\medskip

Now, our aim is to derive from Theorem \ref{harnack-relativistic-vero} a statement of the Harnack inequality which is 
useful for the construction of Harnack chains. First of all, we define the positive cone
\begin{align}	
	\label{paraboloid}
	P_{r} (0)= \Big \{ (p,y,t) \in \R^3 : \, & | p | < t^{\frac{1}{2}}, 
						 | y  | < t^\frac{3}{2}, \, - \theta^2 r^2 \le t < 0 \Big \}.
	\end{align}
Moreover, in analogy with the definition of $H_r^{\mathcal{L}}(z_0)$ and $S_r^{\mathcal{L}}(z_0)$ in \eqref{eq-cyl-B-2}, we set
$ P^{\mathcal{L}}_{r}(z_0) := z_0 \circ_{\mathcal{L}}P_r(0)$. We are now in a position to state the following result, whose proof can be found in \cite[Proposition 3.2]{PaPo06}.
\begin{theorem}
	\label{harnack-2}
	Let $\Omega$ be an open set in $\R^{3}$ containing $H_r(z_0)$ for some $z_0 \in \R^3$ and $r\in \left]0, \frac12\right[$. Then
	\begin{equation*}
		f(z_0 \circ_{\mathcal{L}} z) \le C_H f(z_0) 
	\end{equation*}
	for every non negative solution $f$ of $\mathcal{L} f = 0$ in $\Omega$ and for every $z \in P_{r} (0)$.
	
\end{theorem}
Next, we show that the trajectories defined in \eqref{eq-admiss} belong to a certain positive cone provided a suitable choice of the parameter $s \in [0,T]$.
\begin{lemma} \label{prop-path}
	Let $s \in [0,T]$, $\omega \in L^2([0,T])$ be a control and let $\gamma(s)=(p(s),y(s),t(s))$ be an $\L$-admissible 
	path starting from $z_0=(p_0,y_0,t_0) \in \R^3$. 
	Then for every $r \in \left]0, \frac12 \right[$ there exist two positive constants 
	$k_0:= 2\ln \left( \frac32 \right)$ and $\theta \in ]0,1[$,
	only depending on operator $\L$, such that
	\begin{align*}
		 \gamma(s) \in P^{\mathcal{L}}_{r}(z_0),
	\end{align*}
for every $s \in \left[0,\sqrt{\frac{{2}}{{3}}}\theta^2\,r^2\right]$ such that
	\begin{equation*}
		  \int_0^s | \omega(\tau)|^2 d \tau \le k_0^2.
	\end{equation*}
\end{lemma}

\begin{proof}
Without loss of generality, we fix $z_0=(0,0,0)$ and we give proof of this result for a given $\L$-admissible path starting from $(0,0,0)$. 
The general case directly follows from the translation invariance with respect to the group law \eqref{group-law}. 

Thus, we begin by considering the first component of $\gamma(s)$. In virtue of \eqref{expl-diffeq}, for every $s>0$, we have
\begin{equation*}
	\int\limits_0^s p'(\tau)d\tau = \int \limits_0^s \omega(\tau)\sqrt{p^2(\tau)+1} \, d \tau 
\end{equation*}
and therefore 
\begin{equation}\label{arcsinh}
	 \int \limits_0^s \omega(\tau) \, d \tau  =
	\int\limits_0^s \frac{p'(\tau)}{\sqrt{p^2(\tau)+1}}d\tau=\sinh^{-1}(p(s))=\ln \left(p(s)+\sqrt{p^2(s)+1}\right).
\end{equation}

Now, we apply H\"older's inequality and we estimate the $L^2$ norm of the control with $k_0$ to get
\begin{equation} \label{controllo}
	\left\vert \int \limits_0^s \omega(\tau) \, d \tau \right\vert \leq 
	\left( \int \limits_0^s |\omega(\tau)|^2 \, d \tau \right)^{\frac12} \sqrt{s} \le \,
	k_0 \sqrt{s} \le \ln (1 + \sqrt{s}), \qquad  \forall s \in \left[0,\frac{1}{4}\right],
\end{equation}
We observe that the last inequality follows from our choice of $k_0$ and the concavity of $\ln(1+x)$, which in particular implies that $\ln(1+x) \geq 2 \ln (3/2)x$ for every $x \in [0,1/2]$.
As a consequence,
\begin{equation}\label{e-est}
	\left\vert e^{ \int_0^s \omega(\tau) \, d \tau}-1 \right\vert \leq  
  e^{\left\vert \int_0^s \omega(\tau) \, d \tau \right\vert}-1  
	\, \le \sqrt{s}, \qquad \forall
		 s \in \left[0,\sqrt{\frac{{2}}{{3}}}\theta^2\,r^2\right].
\end{equation}
Then, combining \eqref{arcsinh}, \eqref{controllo} and \eqref{e-est}, we obtain
\begin{equation} \label{sol-v}
	|p(s)|\leq|p(s) +\sqrt{p^2(s)+1} -1| 
	\le  \sqrt{s}, \qquad \forall s \in \left[0,\sqrt{\frac{{2}}{{3}}}\theta^2\,r^2\right].
\end{equation}
Next, we consider the second component of $\gamma(s)$, that is
\begin{equation*}
	y(s) = - \int \limits_0^s p(\tau) \, d \tau .
\end{equation*}
Owing to \eqref{sol-v}, we immediately get
\begin{equation} \label{y-est}
	\left| y(s)  \right| \le	\int \limits_0^s \sqrt{\tau} \, d \tau =\frac23 
	 s^{\frac32}< s^{\frac32},
	\qquad \forall s \in \left[0,\sqrt{\frac{{2}}{{3}}}\theta^2\,r^2\right].
\end{equation}
By combining the above inequality, with \eqref{sol-v}, we obtain 
\begin{align} \label{t-est}	
	0\leq- t(s) =  \int\limits_0^s \sqrt{p^2(\tau)+1} \, d\tau   \le  \int\limits_0^s \sqrt{\tau+1} d\tau  \leq \sqrt{\frac{3}{2}}\, s \leq \theta^2 r^2, \qquad  \forall s \in \left[0,\sqrt{\frac{{2}}{{3}}}\theta^2\,r^2\right].
\end{align} 
Hence,
\begin{equation*}
  	|p(s) | \le  s^{\frac{1}{2}}, \quad 
  	\left|y(s) \right| \le s^{\frac32}, 
 	 \quad  - \theta^2 r^2 \le t(s) \le 0,
	 \qquad  \forall s \in \left[0,\sqrt{\frac{{2}}{{3}}}\theta^2\,r^2\right].
\end{equation*}
This concludes the proof.
\end{proof}

Finally, we are in a position to prove a more suitable statement of the Harnack inequality for points of 
an admissible trajectory. 
\begin{proposition} \label{prop-cor}
	Let $T>0$, $R>0$ and $z_0=(p_0,y_0,t_0)  \in \R^3$.
	Let $s \in [0,T]$, $\omega \in L^2([0,T])$ be a control and let $\gamma(s)=(p(s),y(s),t(s))$ be an $\L$-admissible 
	path starting from $z_0$.
	Then, for every non negative solution $f$ 
	to $ \L f = 0$ in $H^{\mathcal{L}}_{R}(z_0)$,
	there exist three positive constants 
	$k_0:= 2\ln \left( \frac32 \right)$, $C_H$ and $\theta \in ]0,1[$,
	only depending on operator $\L$, such that
	\begin{equation*}
		f(\gamma(s)) \, \le \, C_{H} f(z_0),
	\end{equation*}
	for every $s \in 
		 \left[ 0, \sqrt{\frac{2}{3}}\,\theta^2 r^2 \right]$
	such that
	\begin{equation*}
		  \int_0^s | \omega(\tau)|^2 d \tau \le k_0^2.
	\end{equation*}
\end{proposition}
\begin{proof}
	 The result directly follows by combining Theorem \ref{harnack-relativistic-vero} with Proposition \ref{prop-path}.
\end{proof}

\subsection{Optimal control problem}
We are state and prove an useful intermediate result, which provides us with an estimate for any positive solution $f$ to 
$\L f = 0$ at any point of a given $\L$-admissible path in terms of the $L^2$-norm of the control. Results of this kind are 
usually referred as \textit{non local} Harnack inequalities. In particular, our result is an extension of \cite[Theorem 1.1]{PaPo06} 
and \cite[Proposition 1.1]{BoPo}. For this reason, we hereby report only a sketch of the proof and for further details we refer the reader to \cite{PaPo06, BoPo}.

\begin{proposition}\label{prop-Phi}
	Let $z_0=(p_0,y_0, t_0) \in \R^2 \, \times \, ]T_0, T_1]$ and let $\omega \in L^2([0,T])$ be a control and $\gamma(s)=(p(s),y(s),t(s))$ be the corresponding $\L$-admissible path starting 
	from $z_0=(p_0,y_0,t_0)\in \R^3$. Moreover, let us fix $\,T_0 < t(s) < t_0 < T_1\,$, with $t_0-t(s) \leq \theta^2 (t_0-T_0) \leq \frac{\theta^4}{4}$. 
	Then, for every non negative solution $f$ to $\L f =0$ in $\R^2\, \times \,]T_0, T_1]$, there exist three positive constants 
	$k_0$, $\theta$, $C_H$, only depending on operator $\L$, such that
	\begin{equation}\label{bound-phi}
		f(\gamma(s)) \leq C_H^{\frac{\Phi(\omega)}{k_0^2}+1} f(z_0),
	\end{equation}
where
	\begin{equation} \label{total-cost}
		\Phi(\omega)=\int_0^{s}|\omega(\tau)|^2 d\tau.
	\end{equation}
\end{proposition}

\begin{proof}
	Let $k_0$, $C_H$ and $\theta$ be the constants of Proposition \ref{prop-cor}. 
We first observe that, if  
$$
\int_0^{s} |\omega(\tau)|^2 d\tau \leq k_0^2,
$$ 
then
\begin{equation*}
\gamma(s) \in P^{\mathcal{L}}_{r}(z_0), \quad r:= \sqrt{t_0-T_0} \leq \frac{1}{2},
\end{equation*}
in virtue of Proposition \ref{prop-path} and assumption $t_0-t(s) \leq \theta^2(t_0-T_0)$. Since $H^{\mathcal{L}}_r(z_0) \subset \R^2 \, \times  \,  \times ]T_0,T_1[$ thanks to our choice of $r$, Proposition \ref{prop-cor} can be applied and there holds $f(p,y,t)\leq C_H\, f(z_0)$, where $C_H$ is the constant given by Theorem \ref{harnack-relativistic-vero}.

If the above inequality is not satisfied, we set
\begin{equation}\label{def-k}
k=\max \left\lbrace j \in \mathbb{N}\, : \, \int_0^{s} |\omega(\tau)|^2 d\tau > jk_0^2 \right\rbrace
\end{equation} 

and define recursively a sequence of times starting from $\sigma_0 \equiv 0$ as follows
\begin{equation}\label{sigma-seq}
	\sigma_{j}=\min \Bigg\lbrace s, \, \inf \left\lbrace \sigma >0\, : \, \int_0^\sigma |\omega(\tau)|^2 d\tau > jk_0^2 \right\rbrace \Bigg\rbrace,
\end{equation}
for every $j=1,\ldots,k+1$.
Thanks to \eqref{def-k}, the sequence in \eqref{sigma-seq} ends after a finite number of steps when the upper bound $\sigma_{k+1}\equiv s$ is reached. Moreover, for every $j=0,\ldots,k+1$, we define the sequence $t_j=t(\sigma_j)$, which satisfies $t(s)\equiv t_{k+1} < t_k < t_{k-1}< \ldots t_1 < t_0$. We now observe that 
\begin{equation*}
	H^{\mathcal{L}}_{r_j}(\gamma(\sigma_j)) \subset \R^2 \times ]T_0,T_1[, \quad \textit{for $r_j=\frac{\sqrt{t_{j}-t_{j+1}}}{\theta}$, $j=1,\ldots,k$} .
\end{equation*}
In addition, we clearly have $t_{j}-t_{j+1}\leq \theta^2 \, r_j^2$ and $r_j \leq \frac{1}{2}$, since $\frac{t_j-t_{j+1}}{\theta^2}\leq \frac{t_0-T_0}{\theta^2}\leq \frac{1}{4}$.
Finally, as $\int_0^{\sigma_1}|\omega(\tau)|^2d\tau \leq k_0^2$, we can apply Proposition \ref{prop-cor} and get 
$f(\gamma(\sigma_1)) \leq C_H f(\gamma(0))=C_H f(z_0)$. Similarly, owing to $\int_{\sigma_1}^{\sigma_2}|\omega(\tau)|^2d\tau \leq k_0^2$ and applying again Proposition \ref{prop-cor} to the trajectory steering $(p_1,y_1,t_1):=\gamma(\sigma_1)$ to $(p_2,y_2,t_2):=\gamma(\sigma_2)$, we obtain $f(\gamma(\sigma_2))\leq C_H f(\gamma(\sigma_1))\leq C_H^2 f(z_0) $.
We then iterate the above argument until at step $k+1$ and we obtain
\begin{equation*}
f(\gamma(s))\leq C_H^{k+1} f(z_0).
\end{equation*}
We point out that the points $(\gamma(\sigma_j))_{j=1}^{k}$, chosen along the trajectory 
$\gamma(\cdot)$, define a Harnack chain. Finally, from \eqref{def-k}, it follows that
\begin{equation*}
k < \frac{\int_0^{s} |\omega(\tau)|^2 d\tau}{k_0^2},
\end{equation*}
and this concludes the proof of Proposition \ref{prop-Phi}.
\end{proof}

Estimate \eqref{bound-phi} provides us with a bound dependent on the choice of the $\L$-admissible path steering $z_0$ to $\gamma(s)$. 
Hence, we introduce the \textit{value function}
\begin{equation}\label{defPSI}
\Psi(p_0,y_0,t_0;p_1,y_1,t_1):= \inf_{\omega \in L^2([0,T])}\Phi(\omega),
\end{equation} 
where the infimum is taken over all the $\L$-admissible paths steering $z_0:=(p_0,y_0,t_0) \in \R^3$ to $z_1:=(p_1,y_1,t_1) \in \R^3$.
Then, as a straightforward consequence of Proposition \ref{prop-Phi}, we obtain 
\begin{equation}\label{sol-Psi}
	f(\gamma(s))\leq M^{\frac{\Psi(p_0,y_0,t_0;p(s),y(s),t(s))}{k_0^2}+1} f(z_0),
\end{equation}
whenever $f$ satisfies the assumptions of Proposition \ref{prop-Phi}. 
As it will be clear in the following of this section, equation \eqref{sol-Psi} is a key step in proving the lower bound for the fundamental solution of $\L$.
Thus, in order to characterize the minimizing cost $\Psi$, and hence to obtain the
best exponent in 
\eqref{bound-phi}, we formulate the natural optimal control problem, i.e. we consider the function $\omega$ as the \textit{control} of the path $\gamma$ in \eqref{eq-admiss} and we look for the one 
minimizing the \textit{total cost} $\Phi$ defined in \eqref{total-cost}. 
As observed above, given a solution to \eqref{eq-admiss}-\eqref{eq-admiss-pt}, the same control steers $(0,0,0)$ to $(p_0,y_0,t_0)^{-1} \circ_{\mathcal{L}} (p_1,y_1,t_1)$. As the cost $\Phi$ depends on the control only, the two trajectories have the same cost. Hence,
\begin{equation*}
\Psi(p_0,y_0,t_0;p_1,y_1,t_1)=\Psi\left( 0,0,0;(p_0,y_0,t_0)^{-1} \circ_{\mathcal{L}} (p_1,y_1,t_1)\right).
\end{equation*}
As a consequence, we will fix the initial condition $(p_0,y_0,t_0)=(0,0,0)$ in \eqref{eq-admiss}-\eqref{eq-admiss-pt} and then use the invariance property to solve it with a 
general initial condition. Thus, our aim is to study the optimal control problem
\begin{align} \label{opt-con-pb} \nonumber
	&\inf_{\omega \in L^2([0,T])}\int_0^{T}\omega^2(\tau)d\tau \quad \textit{ subject to the constraint} \\ 
 	&\qquad \qquad \qquad \begin{cases}
 		p'(s)&=\omega(s)\sqrt{p^2(s)+1},  \\
 		y'(s) &=- p(s), \qquad 0\leq s \leq T,\\
 		t'(s) &=-\sqrt{p^2(s)+1},
 	\end{cases} \\ \nonumber
	&\textit{with } \, (p,y,t)(0)=(0,0,0),\quad (p,y,t)(T)=(p_1,y_1,t_1), \quad \textit{with $t_1 <0$}. \nonumber
\end{align}

To solve problem \eqref{opt-con-pb}, one possible approach could be to apply the Pontryagin Maximum Principle (see \cite[Chapter 6]{Vinter}) and to compute the Hamiltonian
\begin{equation}\label{hamiltonian}
\begin{split}
&H(p,y,t,\lambda_1,\lambda_2,\lambda_3,m_0,\omega)=\\
&\qquad\qquad\qquad\lambda_1(s) \omega(s)\sqrt{p^2(s)+1} - \lambda_2(s) p(s)-\lambda_3(s)\sqrt{p^2(s)+1}+m_0 \omega^2(s),
\end{split}
\end{equation}
where $\lambda_1$, $\lambda_2$ and $\lambda_3$ are the coordinates of the covector $\lambda$. 

We recall the first order optimality condition is ever considered to be sufficient, unless the \textit{normality condition} holds, that is 
when the Lagrange multiplier $m_0$ is not vanishing, see \cite{AS95}.
Hence, we first show the normality condition holds true in the case of our interest.

\medskip

\begin{proposition} \label{abn-ext}
	Problem \eqref{opt-con-pb} admits no abnormal extremals.
\end{proposition}
\begin{proof}
We argue by contradiction by assuming $m_0=0 $ in \eqref{hamiltonian}. 
Given this choice of $m_0$, \eqref{hamiltonian} now reads as follows
\begin{equation*}
	H(p,y,t,\lambda_1,\lambda_2,\lambda_3,0,\omega)=\lambda_1(s) \omega(s)\sqrt{p^2(s)+1}- \lambda_2(s) p(s)-\lambda_3(s)\sqrt{p^2(s)+1}.
\end{equation*}
In this case, the maximization of the Hamiltonian reads as follows
\begin{equation*}
\frac{\p H}{\p \omega}(p,y,t,\lambda_1,\lambda_2,\lambda_3,0,\omega)=\lambda_1(s)\sqrt{p^2(s)+1}=0 \quad \Rightarrow \quad \lambda_1(s)=0, \quad \forall s \in [0,T].
\end{equation*}
Moreover, owing to $\lambda_1(s)=0$, for every $s \in [0,T]$ there holds
\begin{equation*}
\lambda_1'(s)=-\frac{\p  H}{\p p}(p,y,\lambda_1,\lambda_2,\lambda_3,0,\omega)=\lambda_2(s)+\frac{\lambda_3(s)p(s)}{\sqrt{p^2(s)+1}}=0 \quad \Rightarrow \quad \lambda_2(s)=-\frac{\lambda_3(s)}{\sqrt{p^2(s)+1}}.
\end{equation*}
Additionally, as $\lambda_2'(s)=-\frac{\p H}{\p y}=0$ and $\lambda_3'(s)=-\frac{\p H}{\p t}=0$, we directly compute $\lambda_2'(s)$ and we obtain
\begin{equation*}
\lambda_2'(s)=-\frac{\lambda_3(s)}{\left( p^2(s)+1\right)^{3/2}}=0 \quad \Rightarrow \quad \lambda_3(s)=0\quad \Rightarrow \quad \lambda_2(s)=0, \quad \forall s \in [0,T].
\end{equation*}
Thus, we conclude that 
\begin{equation*}
(\lambda_1(s),\lambda_2(s),\lambda_3(s),m_0)=(0,0,0,0), \quad \forall s \in [0,T],
\end{equation*}
which contradicts the fact that $(\lambda_1(s),\lambda_2(s),\lambda_3(s),m_0)$ is never vanishing. 
\end{proof}

Since no abnormal extramals occur, we choose $m_0=-\frac{1}{2}$ and we compute the optimal control $\omega^*$ as the unique minimizer of $H(p,y,t,\lambda_1,\lambda_2,\lambda_3,-\frac{1}{2},\omega)$, i.e.
\begin{equation}\label{opt-control}
\omega^*(s)=\lambda_1(s)\sqrt{p^2(s)+1}.
\end{equation}
As a consequence, the maximized Hamiltonian $H^*$ is
\begin{equation}
H^*(p,y,t,\lambda_1,\lambda_2,\lambda_3,-\frac{1}{2},\omega^*)=\frac{1}{2}\lambda^2_1(s) (p^2(s)+1)- \lambda_2(s) p(s)-\lambda_3(s)\sqrt{p^2(s)+1},
\end{equation}
and the corresponding Hamiltonian system reads as follows
\begin{equation}\label{hamil-system}
 	\begin{cases}
 		p'(s)&=\lambda_1(s)\left(p^2(s)+1\right),   \\
 		y'(s) &= -p(s),\\
 		t'(s)&=-\sqrt{p^2(s)+1}\\
 		\lambda_1'(s)&=-p(s)\lambda_1^2(s)+\lambda_2(s)+\frac{\lambda_3(s)}{\sqrt{p^2(s)+1}},\\
 		\lambda_2'(s)&=0\\
 		\lambda_3'(s)&=0.
 	\end{cases}
\end{equation}
We observe that, from the last equation in \eqref{hamil-system}, it follows  
$$
	\lambda_2(s)=c_2, \, \, \lambda_3(s)=c_3, \qquad \forall s \in [0,T].
$$ 
Moreover, we choose the parameter $k:=\lambda_1(0)$ as the initial condition for the first extremal, which is the unique solution to \eqref{hamil-system}, with initial condition 
$$
	(p,y,t,\lambda_1,\lambda_2,\lambda_3)(0)=(0,0,0,k,c_2,c_3).
$$ 
Furthermore, as the Hamiltonian is a constant of motion, we set
\begin{equation}
E:=\lambda_1^2(s)\left( p^2(s)+1\right)-2\lambda_2(s){p(s)}-2\lambda_3(s)\sqrt{p^2(s)+1}=k^2-2c_3.
\end{equation}
Moreover, in virtue of \eqref{opt-control} and equations $y'(s)=-p(s)$, $t'(s)=-\sqrt{p^2(s)+1}$, we can compute the cost for extremals as follows
\begin{equation}
\begin{split}
C(\omega(\cdot))&=\int_0^{T}\omega^2(\tau)d\tau=\int_0^{T}\lambda_1^2(\tau)\left(p^2(\tau)+1 \right)
d\tau\\
&=\int_0^{T}\left( E-2c_2 y'(\tau)-2c_3t'(\tau)\right) d\tau =ET-2c_2y_1-2c_3 t_1.
\end{split}
\end{equation}

\begin{remark}
Since solving analytically \eqref{hamil-system} is a real challenge, we are not able to further proceed in our characterization of the optimal control. 
For this reason, the statement of Theorem \ref{boundsL} explicitly reports the value function $\Psi$. In the future, it will be interesting to study this problem from a numerical perspective, as already proposed in \cite{PaPo} for the pricing problem for Asian Options.   
\end{remark}

\setcounter{equation}{0}\setcounter{theorem}{0}
\section{Proof of Theorem \ref{boundsL}} \label{setting}
In this section we prove a lower bound for the fundamental solution $\Gamma$ of $\L$.  We follow the approach proposed in \cite{CPR}, where an analogous result is proved about an operator arising in Finance. A key tool in this argument is a lower bound for a Green function $G$ for operator $\widetilde{\K}$ introduced in \eqref{kolmo}. 

First of all, we consider the functions $a$ and $b$ in \eqref{def-ab} and we modify them for $|x| > \frac12$ in order to have ontinuous coefficients satisfying assumption \textbf{(H)}. It is sufficient to set
\begin{equation} \label{def-ab-modif}
\begin{split}
a(x,y,t) & =  \left(1 - x^2\right)^{5/2} \quad \text{and} \quad b(x,y,t) = -2x \left(1 - x^2\right)^{3/2}, 
\quad \text{for} \quad  - \tfrac{1}{2} \le x \le \tfrac{1}{2}, \\
a(x,y,t) & =  \left(\tfrac{3}{4} \right)^{5/2} \quad \text{and} \quad b(x,y,t) = -\textrm{sign}(x) \left(\tfrac{3}{4}\right)^{3/2}, 
\quad \text{for} \quad  |x| \ge \tfrac{1}{2}.
\end{split}
\end{equation}
Then, \cite[Theorem 1.1]{PolidoroParametrix} provides us with a fundamental solution $\Gamma_{\widetilde{\K}}$  
of the Kolmogorov operator $\widetilde{\K}$ introduced in \eqref{kolmo}. We are now in a position to define a Green function for operator $\widetilde{\K}$ in a suitable cylinder $\mathcal{H}$ defined as follows. 
\begin{equation*}
	\mathcal{H} = \mathcal{S} \times ]0,T[, \qquad {\text{with }}\quad\mathcal{S} = B((1,0), 3/2)  \cap  B((-1,0), 3/2),
\end{equation*}
where $B((x_0,w_0), r)$ denotes the the Euclidean ball of $\R^2$ centered at $(x_0,w_0)$ and of radius $r$, and $T$ is a positive constant. In \cite[Section 4]{PODF} it is proved that the Dirichlet problem for $\widetilde{\K}$ is well-posed on $\mathcal{H}$, i.e. for every bounded continuous function $g$ defined on $\mathcal{H}$ and for every bounded continuous function $\varphi$ defined on $\partial \mathcal{H}$, there exists a unique classical solution $f$ to equation $\widetilde{\K} f = g$ in $\mathcal{H}$. Moreover, $f$ attains continuously the boundary condition at every point of the parabolic boundary $\p_P \mathcal{H}$ of $\mathcal{H}$, that is
\begin{align*}
	\p_P \mathcal{H} = ( \mathcal{S} \times \{ 0 \} ) \cup ( \p \mathcal{S} \times [0,T]).
\end{align*}
The Green function for $\widetilde{\K}$ on $\mathcal{H}$ is defined as the function $G: \overline{\mathcal{H}} \times \mathcal{H} \to [0, + \infty[$ such that
\begin{align*}
	G(x,y,t; \xi, \eta, \tau) := \Gamma_{\widetilde{\K}} (x,y,t; \xi, \eta, \tau) - h(x,y,t; \xi, \eta, \tau),
\end{align*}
where $h(x,y,t; \xi, \eta, \tau)$ is the solution to the Dirichlet problem:
\begin{align} \label{dirichlet-h}
	\begin{cases}
		\widetilde{\K} f = 0 \qquad &\text{in } \mathcal{H},\\
		f = \Gamma_{\widetilde{\K}} (x,y,t; \xi, \eta, \tau) \qquad &\text{in } \p_P \mathcal{H}.
	\end{cases}
\end{align}
We now recall the most important property of function $G$. For every $g \in C^{\infty}_0(\mathcal{H})$ and $\phi \in C^{\infty}_0(\mathcal{S})$, the function
\begin{align*}
	v(x,y,t) := \int \limits_{\mathcal{H}} G(x,y,t; \xi, \eta, \tau) g(\xi, \eta, \tau) \, d \xi \, d \eta \, d \tau +
			 \int \limits_{\mathcal{S}} G(x,y,t; \xi, \eta, 0) \phi(\xi, \eta) \, d \xi \, d \eta
\end{align*}
is a classical solution to the Dirichlet problem 
\begin{align} \label{property-G}
	\begin{cases}
		\widetilde{\K} f = - g \qquad &\text{in } \mathcal{H},\\
		f = \phi \qquad &\text{in } \mathcal{S} \times \{ 0 \}, \\
		f=0 \qquad &\text{in } \p \mathcal{S} \times [0,T]. 
	\end{cases}
\end{align}
We point out that the above property is stated in \cite[Section 4]{PODF} only for $\varphi = 0$.  The validity of \eqref{property-G} follows from well-known properties of the fundamental solution $\Gamma_{\widetilde{\K}}$. We finally recall the statement of a local lower bound for the Green function given in \cite[Theorem 4.3]{PODF}.
\begin{lemma} \label{lower-green}
There exists two positive constants $c_G>0$ and $\delta_G \in ]0, 1]$, only depending on the constants appearing in assumption \textbf{(H)}, such that 
	\begin{equation*}
		G(0, 0, t; 0, 0, 0) \ge \frac{c_G}{t^{2}}, \quad \forall t \in \left] 0, \delta_G \right[.
	\end{equation*}
\end{lemma}

We are now in a position to prove a local lower bound for the fundamental solution $\Gamma$ of the relativistic operator $\L$.
The proof of this result is an adaptation of \cite[Lemma 4.3]{CPR} to the case of our interest.
\begin{lemma} \label{lemma-gamma}
For every positive constant $T$, there exists a psitive constant $\kappa_T$,
only depending on the constants appearing in assumption { \textbf{(H)}}, such that
\begin{equation*}
		\Gamma (0,0, t; 0, 0, 0) \ge \frac{\kappa_T}{t^2} , \qquad \forall t \in \left]0, T \right[.
\end{equation*}	
\end{lemma}
\begin{proof}
In order to prove our claim, we just need to show that there holds
\begin{equation}\label{gammagreen}
		\Gamma (p,y,t; \xi, \eta, \tau) \ge G(p,y,t; \xi, \eta, \tau) \qquad \forall (p,y,t; \xi, \eta, \tau) \in \overline{ \mathcal{H}} \times  \mathcal{H}.
\end{equation}
Indeed, if \eqref{gammagreen} holds true, then the result for $0 < t < \delta_G$ is a straightforward consequence of Lemma \ref{lower-green}. The result for any $T > \delta_G$ follows from the fact that $\Gamma$ is a continuous strictly positive function.

Thus, it is only left to prove inequality \eqref{gammagreen}. To this end, for every non-negative $\phi \in C^{\infty}_0(\mathcal{S})$ and for every $(p,y,t) \in \overline{\mathcal{H}}$, we set
\begin{align*}
	v(p,y,t) &:= \int \limits_{\mathcal{S}} G(p,y,t; \xi, \eta, 0) \phi(\xi, \eta) \, d \xi \, d \eta, \\
	u(p,y,t) &:= \int \limits_{\mathcal{S}} \Gamma(p,y,t; \xi, \eta, 0) \phi (\xi, \eta) \, d \xi \, d \eta,
\end{align*}
where $\Gamma$ is the fundamental solution of $\L$ and $G$ is the Green function of $\widetilde{\K}$ in $\mathcal{H}$.
By Definition \ref{fun-sol-def} and \eqref{property-G}, both $v$ and $u$ are solution to $\L f = 0$ in $\mathcal{H}$, 
or equivalently to $\widetilde{\K} f = 0$.
Then by \eqref{property-G} and comparison principle we find $u \ge v$ in $\mathcal{H}$. Hence, this implies
\begin{align*}
	\int \limits_{\mathcal{S}} \left( \Gamma (p,y,t; \xi, \eta, 0) - G(p,y,t; \xi, \eta, 0) \right) \phi (\xi, \eta) \, d \xi \, d \eta \ge 0
	\end{align*}
for every non-negative $\phi \in C^\infty_0(\mathcal{S})$ and for every $(p,y,t) \in \overline{\mathcal{H}}$. This concludes the proof. 
\end{proof}

\bigskip

\noindent
\textbf{Proof of Theorem \ref{boundsL}.}
By choosing $T_0=0$ and $T= T_1=t_0$, we apply Proposition \ref{prop-Phi} and Lemma \ref{lemma-gamma}
and obtain
\begin{align*}
	\Gamma (p_0,y_0,t_0; 0,0,0) 
	&\ge C_H^{- \frac{\Psi(p_0,y_0,t_0; 0,0,(1 - \theta^2)t_0)}{k_0^2}-1} \Gamma \left(0,0, (1 - \theta^2)t_0; 0, 0, 0 \right)  \\
	&\ge C_H^{- \frac{\Psi(p_0,y_0,t_0; 0,0,(1 - \theta^2)t_0)}{k_0^2}-1} \frac{\kappa_T }{(1 - \theta^2)^2t_0^{2}}
\end{align*}
for every $(p_0,y_0,t_0) \in \R^3$ such that 
$t_0 \leq \frac{\theta^2}{2}$. This proves Theorem \ref{boundsL} for $(p_1, y_1, t_1)=(0,0,0)$, where 
$$
	c_{T}=  C_H^{-1} \frac{\kappa_T}{(1 - \theta^2)^2}.
$$ 
The statement for a general point $(p_1, y_1, t_1) \in \R^3$ follows from the traslation invariance of $\L$ 
with respect to \eqref{group-law}.

\appendix
\section{Higher dimensional case}\label{sec-Horm}
\subsection{H\"ormander's operators}
In this section, we check that the $d$-dimensional operator $\L$ in \eqref{oprel} writes in the form \eqref{hormander-op} and satisfies the H\"ormander's condition \eqref{hormander}.

We first explain how to choose the vector fields $X_1, \dots, X_d$ in \eqref{vector-fields}. As a first step, we observe that equation \eqref{oprel} can be written in its non-divergence form
\begin{equation} \label{oprel-nd}
    \L f (p,y,t) =  \textrm{Tr} \left( \left( \mathbb{I}_d+ p\otimes p \right) \nabla^2_p f\right) + 
    d \, p \cdot \nabla_p f - Y f = 0.
\end{equation}

We consider the $d \times d$ symmetric matrix $\mathbf{X}$ 
\begin{equation*} \label{matrix}
    \mathbf{X}(p) = \left( X_1(p), \dots, X_d(p) \right),
\end{equation*}
whose columns are the coefficients of the vector fields $X_1, \dots, X_d$.
We have 
\begin{equation*}
	\sum_{j=1}^d X_j f = \mathbf{X} \, \nabla_p f, \qquad \text{and} \qquad
	\sum_{j=1}^d X_j^2 f = \mathbf{X}^2 \, \nabla^2_p f + \tilde c \cdot \nabla_p f,
\end{equation*}
for some vector $\tilde c = \tilde c(p)$. We then determine $\mathbf{X}$ such that $\mathbf{X}^2 = \mathbb{I}_d+ p \otimes p$. To do this, we recall that, for any given $q \in \R^d$, we have 
\begin{equation*} 
    \left( \mathbb{I}_d+ q \otimes q \right)^2 = \mathbb{I}_d+ (2 + |q|^2) \, q \otimes q.
\end{equation*}
Then, 
\begin{equation} \label{id+oqquadro}
    \left( \mathbb{I}_d+ q \otimes q \right)^2 = \mathbb{I}_d+ p \otimes p 
\end{equation}
if we choose 
\begin{equation} \label{q=alpha(p)}
    q = \alpha \, p \qquad \text{for some $\alpha$ such that} \qquad (2 + |q|^2) |q|^2 = |p|^2.
\end{equation}
Direct computations show that the second equality in \eqref{q=alpha(p)} implies that
\begin{equation} \label{relq}
    1 + |q|^2 = \sqrt{|p|^2+1} 
\end{equation}
and therefore
\begin{equation}\label{alpha}
\alpha = \tfrac{1}{\sqrt{1 + \sqrt{|p|^{2} +1}}}.
\end{equation}
Hence, by choosing
\begin{equation*} 
 \mathbf{X} = \mathbb{I}_d+ q \otimes q, \qquad 
 q = \tfrac{1}{\sqrt{1 + \sqrt{|p|^{2} + 1}}} \, p
\end{equation*}
we find the vector fields $X_1, \dots, X_d$ introduced in \eqref{vector-fields}. Moreover, the components of $\tilde c$ are
\begin{equation} \label{vector-fields-q}
 \tilde c_j(p) = \sum_{i,k=1}^{d} \left( \delta_{ik} +q_i q_k \right) \tfrac{\p (q_j q_k) }{ \p p_i}, \quad j=1, \dots, d. 
\end{equation}
Thus, from \eqref{oprel-nd} and \eqref{id+oqquadro} we obtain the following identity
\begin{equation*} 
    \L =  \sum_{j=1}^d X_j^{2} + (d \, p - \tilde c(p)) \cdot \nabla_p - Y.
\end{equation*}
In order to conclude that $\L$ writes in the form \eqref{hormander-op}, we observe that the matrix $\mathbb{I}_d - \tfrac{1}{1 + |q|^2} \, q \otimes q$ is the inverse of $\mathbb{I}_d+ q \otimes q$.
As a consequnce, we have 
\begin{equation*} 
    \left( \mathbb{I}_d - \tfrac{1}{1 + |q|^2} \, q \otimes q \right) \left( X_1(p), \dots, X_d(p) \right) = \nabla_p.
\end{equation*}
This concludes the proof of \eqref{hormander-op}, where the vector $c(p) = \left(c_1(p), \dots, c_d(p) \right)$ is defined as
\begin{equation*} 
    c(p) = (d \, p - \tilde c(p))^T \left( \mathbb{I}_d - \tfrac{1}{1 + |q|^2} \, q \otimes q \right) 
\end{equation*}
and has smooth coefficients, in virtue of \eqref{q=alpha(p)} and \eqref{alpha}.

\medskip

\medskip

We next prove that $\L$ does satisfy the H\"ormander's condition \eqref{hormander}. We first note that the Lie algebra generated by $X_1, \dots, X_d, X_{d+1}$ agrees with the Lie algebra generated by $X_1, \dots, X_d, Y$, and then we claim that
\begin{equation}\label{hormander-Y}
	{\rm rank}\, {\rm Lie} \, \lbrace X_1,\ldots, X_d ,Y \rbrace \,(p,y,t)=2d+1, 
	\qquad \forall (p,y,t) \in \mathbb{R}^{2d+1}.
\end{equation} 
We compute the commutator $[X_j,Y]$ for $j=1, \dots, d$. We find that
\begin{equation*}
	[X_j, Y] f := X_j Y f - Y X_j f = \sum_{k=1}^{d} \left( \delta_{jk} + q_j q_k \right) \tfrac{\partial f}{\partial y_k} 
	+ p_j \tfrac{\partial f}{\partial t}.
\end{equation*}
We now consider the $(2d+1) \times (2d+1)$ matrix $\mathbf{M}$ whose columns are the coefficients of $X_1, \dots, X_d$, $[X_1, Y], \dots, [X_d, Y], Y$ and we prove that 
\begin{equation} \label{eclaim}
 \textrm{det} \, \mathbf{M} = \sqrt{|p|^{2}+1}.
\end{equation}
We have
\begin{align*}
\mathbf{M} = \Big( X_1, \dots, X_d, [X_1, Y], \dots, [X_d, Y], Y \Big) = 
\begin{pmatrix}
 \mathbb{I}_d + q \oplus q & \mathbb{O}_d & 0_d \\
 \mathbb{O}_d  & \mathbb{I}_d + q \oplus q &  p \\
 0_d^T &  p^T  & \sqrt{|p|^{2}+1}
\end{pmatrix},
\end{align*}
where $\mathbb{O}_d$ is the $d \times d$ matrix whose entries are zeros, and $0_d$ is the zero column vector of $\R^d$ 
Up to a change of basis in $\R^d$, it is not restrictive to assume that $q = |q| \, \mathbf{e}_d$, being $\mathbf{e}_d$ the $d$-th vector of the canonical basis of $\R^d$. Then the matrix ${\mathbf{M}} = $ takes the simpler form
\begin{equation*}
{\mathbf{M}} =  \begin{pmatrix}
 \mathbf{D} & 0_{2d-1} & 0_{2d-1} \\
 0_{2d-1}^T  & 1 + |q|^2 &  |p| \\
 0_{2d-1}^T & |p|  & \sqrt{ |p|^{2}+1}
\end{pmatrix},
\end{equation*}
where $\mathbf{D} = \mathbb{I}_{2d-1} + (1 + |q|^2) \, \mathbf{e}_d \otimes \mathbf{e}_d$. Thus, \eqref{eclaim} follows from the first equality in \eqref{alpha}.
\medskip

\subsection{Lorentz invariance}
The invariance with respect to Lorentz transformations is also preserved in the higher dimensional case. Indeed, it is sufficient to observe that the diffusion operator in \eqref{oprel} is the Laplace-Beltrami operator over the Riemannian manifold $\left( \R^d,g \right)$, where $g$ is the metric induced by the Minkonwski metric over the hyperboloid $\mathfrak{g}=	\lbrace (E,p): E= \sqrt{|p|^2+1} \rbrace$. We recall that the Laplace-Beltrami operator is invariant with respect to isometries. Then, the invariance with of $\L$ follows from the fact that the Lorentz transformation in the momentum component corresponds to a translation over $\mathfrak{g}$.  Moreover, the invariance of the drift term $Y$ in \eqref{oprel} follows immediately from \eqref{comp-drift}, which clearly still holds true in the higher dimensional case.

\section*{Aknowledgments}
The authors are members of “Gruppo Nazionale per l’Analisi Matematica,
la Probabilità e le loro Applicazioni (GNAMPA)” of Istituto Nazionale di Alta Matematica
(INdAM).

\end{document}